\documentclass[11pt]{amsart}
\usepackage{amssymb}
\usepackage{amsfonts}
\usepackage{latexsym}
\usepackage{amscd}
\usepackage[mathscr]{euscript}

\usepackage[active]{srcltx}

\usepackage{enumerate}

\usepackage[utf8]{inputenc}
\usepackage[all,cmtip,2cell]{xy}

\title{Stable elements and Property (S)}
\usepackage{multicol}

\author{Joan Bosa}
\address{Dpt. de Matemàtiques, Fac. C, Universitat Autònoma de Barcelona, 08193 Bellaterra, Barcelona (Spain)}
\email{jbosa@mat.uab.cat}

\theoremstyle{plain}
\newtheorem{lemma}{Lemma}[section]
\newtheorem{theorem}[lemma]{Theorem}
\newtheorem{corollary}[lemma]{Corollary}
\newtheorem{proposition}[lemma]{Proposition}
\newtheorem{definition}[lemma]{Definition}
\newtheorem{question}[lemma]{Question}
\newtheorem*{proposition*}{Proposition}
\newtheorem*{theorem*}{Theorem}
\newtheorem*{definition*}{Definition}
\newtheorem*{claim*}{Claim}
\newtheorem*{notation*}{Notation}

\newtheorem{remark}[lemma]{Remark}

\newcommand{\Cu}{\mathrm{Cu}}

\newcommand{\ra}{\rangle}
\newcommand{\la}{\langle}

\usepackage{mathtools}

\thanks{The author was partially supported both by the DGI-MINECO and European Regional Development Fund through the grant MTM2017-83487-P and by the Generalitat de Catalunya through the grant 2017-SGR-1725. Also, this research was carried out thanks to the Beatriu de Pinós postdoctoral programme of the Government of Catalonia's Secretariat for Universities and Research of the Ministry of Economy and Knowledge.}
\date{\today}

\begin{document}
	\begin{abstract}
	We study the relation (and differences) between stability and Property (S) in the simple  and stably finite framework. This leads us to characterize stable elements in terms of its support, and study these concepts from different sides :  hereditary subalgebras, projections in the multiplier algebra and order properties in the Cuntz semigroup. We use these approaches to show both that cancellation at infinity on the Cuntz semigroup just holds when its Cuntz equivalence is given by isomorphism at the level of Hilbert right-modules, and  that different notions as Regularity, $\omega$-comparison, Corona Factorization Property, property R, etc.. are equivalent under mild assumptions.
	\end{abstract}
	\maketitle
\section*{introduction}	
Characterizing and understanding the class of separable simple C*-algebras is an important problem since the origins of operator algebra theory. Concerning structural questions, one of the most natural (and old) problems wonders how to determine stability for separable and simple C*-algebras. Recall that a separable C*-algebra $A$ is said to be stable if $A\cong A\otimes\mathcal K$, where $\mathcal K$ is the C*-algebra of compact operators on a separable, infinite dimensional Hilbert space. Furthermore, we say that a positive element $a\in A$ is stable if its associated hereditary subalgebra $\overline {aAa}$ is stable.

Related answers of the above fundamental question were given by Brown, Cuntz, Hjelmborg, R\o rdam, Winter and many others (\cite{Cu78,Bro88,HRW07, HR98, OPR2, Ro97}). Indeed, in \cite{HR98} the authors shown that an $AF$-algebra is stable if and only if it admits no bounded traces. Moreover, they wondered if a C*-algebra $A$ is stable if and only if  satisfies what is known as property (S) ($A$ admits no bounded 2-quasitrace and no quotient of $A$ is unital). R\o rdam answers negatively the last question in \cite{Ro97}, where he built the first example of a C*-algebra $B$ such that $M_2(B)$ is stable, but $B$ is not stable. Hence, stability is not an stable property.

 In the literature, there exist several properties that relate property (S) and stability. We say that a C*-algebra $A$ is regular (asymptotically regular)  if any full subalgebra $D$ of $A\otimes \mathcal K$ satisfying property (S) is itself stable (there exists $n\geq 1$ such that $M_n(D)$ is stable). The notion of regularity was coined by R\o rdam in \cite{RoSt} to study stable C*-algebras, and it is equivalent to both a pure algebraic condition known as $\omega$-comparison for the Cuntz semigroup of $A$ (see \cite{OPR2}), and a more analytical property concerning projections of $\mathcal M(A\otimes \mathcal K)$. Indeed, it is shown in Proposition \ref{prop:PropS} that a multiplier projection $P$ belongs to a non-regular ideal of $\mathcal M(A\otimes\mathcal K)$ if and only if the hereditary C*-subalgebra $P(A\otimes\mathcal K)P$ satisfies property (S).

Another important property for C*-algebras linked to stability is the Corona Factorization Property (CFP). This was introduced by Elliott-Kucerovsky in \cite{EK01} in order to study absorbing extensions, and it is one of the natural candidates to determine "nice" C*-algebras. The CFP was used by Zhang (and later by Ortega-Perera-R\o rdam in \cite{OPR2}) to get a dichotomy result (simple C*-algebras with real rank zero and CFP are either stably finite or purely infinite). There exist several equivalent definitions of CFP (see for instances \cite{BP18, NgCFP, OPR1, RobertRordam}). We take as the definition the following property, which is shown to be equivalent to the CFP in \cite[Theorem 4.2]{KN06} and \cite[Theorem 5.13]{OPR2}. A separable C*-algebra $A$ satisfies the CFP if given a full, hereditary subalgebra $D$ of $A\otimes \mathcal K$ such that $M_n(D)$ is stable for some $n\geq 1$, then $D$ itself is stable.

Focusing on simple, stably finite and separable C*-algebras, we study all the above mentioned properties looking at the  Cuntz semigroup invariant associated to C*-algebras. In this setting, the Cuntz semigroup has a unique properly infinite element, which we usually denote by $\infty$ (see \cite{BP18} for further details in this direction). The $\infty$ element, which coincides with the largest element of the Cuntz semigroup, has different properly infinite representatives in $(A\otimes\mathcal K)_+$ (a natural representative of $\infty$ is a strictly positive element of $A\otimes\mathcal K$); however, one wonders whether all of the other representatives of $\infty$ are also stable elements. A deep study of (properly) infinite elements was run by Kirchberg-R\o rdam, and they explained the main differences between these elements in \cite{KR00}. Building up from there, in Lemma \ref{lem:Stable} of the current paper, stable elements are characterized showing that this notion is equivalent to ask that any small portion at the beginning of its spectrum is properly infinite. Using this description, we show in Theorem \ref{Thm:stableCSE} that asking for all properly infinite elements to be stable is equivalent to an algebraic cancellation property of the Cuntz semigroup known as cancellation of small elements at infinite. In particular, Theorem \ref{Thm:stableCSE} shows that weak cancellation property for the Cuntz semigroup just holds when the Cuntz equivalence of positive elements is induced by isomorphism of the associated Hilbert right-modules. Furthermore, we use Lemma \ref{lem:Stable} and Theorem \ref{Thm:stableCSE} to study $(\omega,n)$-decomposable elements as defined in \cite{RobertRordam}. It is important to point out that the relation between these elements and stability was already stated in \cite[Proposition 9.7]{RobertRordam}; however, the proof of the implication {\rm (i)}$\Rightarrow${\rm (ii)} shown in \cite{RobertRordam} is not correct. We are able to clarify this relationship in Lemma \ref{Lem:xdecomposable}.

As a result of a more analytical approach, it was shown by Brown in \cite{Bro88} that an element $a$ in $A\otimes\mathcal K$ is stable if and only if the multiplier projection associated to its hereditary subalgebra is Murray-von Neumann equivalent to the unit of $\mathcal M(A\otimes\mathcal K)$. Combining this result with our study of property (S) displayed in Proposition \ref{prop:PropS}, we show in Theorem \ref{Thm:RAssym} that asymptotic regularity is equivalent to property R (if a projection is contained in a proper ideal of $\mathcal M(A\otimes\mathcal K)$, it is also contained in a regular ideal). As application of  Theorem \ref{Thm:RAssym},  we show in Corollary \ref{Cor:Dich} that both property $R$ and asymptotic regularity imply dichotomy between stably finite and purely infinite C*-algebras in the simple setting.

We briefly outline the contents of this paper. In the first section we recall notation and some background needed to understand the sequel. In order to ease reading this section, it is divided in two parts: we first provide the necessary knowledge about $\Cu$-semigroups, and we subsequently recall the comparison properties such as $\omega$-comparison and CFP used in the current paper. Section \ref{Sec:Stable} is dedicated to stable elements. In particular, we state the characterization of stable elements explained along the introduction and some of its implications. We finish this section studying the notion of $(\omega,n)$-decomposable elements previously introduced in \cite{RobertRordam}. In section \ref{sec:asymomega} we develop all the machinery needed to get Theorem \ref{Thm:RAssym}. This part was initiated some time ago as a collaboration with M. Christensen, and some results in it already appeared in his PhD-thesis \cite{ChTh}; though, they have never been published in an article before. 

\section{Notation and Preliminaries}\label{Sec:Preliminaries}

\subsection{Cuntz semigroup}
In this part we recall the main facts about the Cuntz semigroup that will be used along the sequel. This semigroup has been deeply studied in the last years due to its relation to the classification programme of separable simple nuclear C*-algebras. We encourage the reader to look at the overview article \cite{APT} for further details on this semigroup. Let's start by its definition.

\begin{definition}\label{defCu}
	Let $A$ be an stable C$^*$-algebra, i.e. $A\cong A\otimes\mathcal K$ (where $\mathcal K$ denotes the C*-algebra of compact operators in an infinite dimensional Hilbert space), and let $a$, $b\in A_+$. We say that $a$ is \emph{Cuntz subequivalent} to $b$, in symbols $a\precsim b$, provided there is a sequence $(x_n)$ in $A$ such that $x_nbx_n^*$ converges to $a$ in norm, i.e. $\|a-x_nbx_n^*\|\to 0$. We say that $a$ and $b$ are \emph{Cuntz equivalent} if $a\precsim b$ and $b\precsim a$, and in this case we write $a\sim b$. Considering this equivalence relation in $A\otimes \mathcal K$, one obtains the abelian semigroup $\Cu(A):=(A\otimes \mathcal K)_+/\!\!\sim$. We denote the equivalence classes by $\la a\ra$. 
	
	This is a partially ordered abelian semigroup where the operation and order are given by $$\la a\ra+\la b\ra=\la \left(\begin{matrix} a & 0 \\ 0 & b \end{matrix}\right)\ra=\la a\oplus b\ra,\,\,\,\,\,\,\la a\ra\leq \la b\ra\text{ if }a\precsim b.$$
	In particular, the semigroup $\Cu(A)$ is referred to as the \emph{Cuntz semigroup}.
\end{definition}

As explained in \cite{BTZ18, CEI08, ORT}, among others, the Cuntz semigroup of a separable C*-algebra $A$ can be described via different frameworks, i.e. positive elements, hereditary subalgebras, open projections in $A^{**}$ and projections in the multiplier algebra $\mathcal M(A\otimes \mathcal K)$. We will use all these settings along the sequel, so let us shortly recall the main facts and notation.

Any positive element in a C*-algebra $A$ naturally defines: an hereditary subalgebra $A_a:=\overline{aAa}$, an open projection $p_a:={\rm (strong) }\,\lim a^{1/n}$ in $A^{**}$, and a Hilbert $A$-module $E_a:=\overline{aA}$ . One can define the Cuntz semigroup of a C*-algebra $A$ looking at each of these different pictures (see \cite{ORT} for further details). In the Hilbert $A$-modules picture, for two Hilbert $A$-modules $E,F$, we write $E\Subset F$ if there exists $x\in\mathcal K(F)$, the compact operators of $\mathcal L(F)$, such that $xe=e$ for all $e\in E$. We use it to define the Cuntz subequivalence between $E$ and $F$, written as $E\precsim_\Cu F$, as if for every Hilbert $A$-submodule $E'\Subset E$ there exists $F'\Subset F$ with $E'\cong F'$ (isometric isomorphism). It is important to notice that, given positive element $a,b\in A_+$, it follows that $E_a\cong E_b$ if and only $A_a\cong A_b$. However, the equivalence relation induced by isomorphism on Hilbert $A$-modules is stronger than the Cuntz equivalence just defined (see \cite{BC09} for a concrete counterexample). It is known that under the extra assumption of stable rank one, the  Cuntz relation is equivalent to the equivalence relation induced by isomorphism on Hilbert $A$-modules.

Moving to the algebraic framework, the Cuntz semigroup of a C*-algebra always belongs to an algebraic category of ordered complete semigroups called $\Cu$. The objects of this category are called $\Cu$-semigroups, and we usually denote them by $S$ and $T$ (see \cite{APT14} for further details). 

Fixing $S$ a $\Cu$-semigroup, let us remind some of its main properties. Every increasing sequence has a supremum  in $S$, and $S$ has an auxiliary relation usually denoted by $\ll$, and called way-below. In particular, one writes \textbf{$x\ll y$} if, whenever $\{x_{n}\}$ is an increasing sequence  satisfying
$y\leq \sup\,x_{n}$, then $x\leq x_{n}$ for some $n$. An element $x$ in $S$ is called \textbf{compact}, if $x\ll x$. We write \textbf{$y<_{s}x$} if there exists $k\in \mathbb{N}$ such that $(k+1)y\leq kx$, and $y\propto x$ if there exists $n\in\mathbb N$ such that $y\leq nx$. 
An element $x$ in $S$ is called \textbf{full} if for any $y',y\in S$ with $y'\ll y$, one has $y'\propto x$, denoted by $y\,\bar{\propto}\, x$.
A sequence $\{x_{n}\}$ in $S$ is said to be \textbf{full} if it is increasing and for any $y',y\in S$ with $y'\ll y$, one has $y'\propto x_{n}$ 
for some (hence all sufficiently large) n. Notice that if $x\in S$ is an order unit, it is also a full element, but the reverse is not true.

Using these notions, we say that $S$ is \textbf{simple} if $x\,\bar{\propto}\, y$ for all nonzero $x,y\in S$. In other words, $S$ is simple if every nonzero element is full.

Focusing on the behaviour of elements in $S$, let us recall that an element $a\in S$ is {\bf finite} if for every element $b\in S$ such that $a+b\leq a$, one has $b=0$. An element is \textbf{infinite} if it is not finite. An infinite element $a\in S$ is {\bf properly infinite} if $2a\leq a$. We say that $S$ is {\bf stably finite} if an element $a\in S$ is finite whenever there exists $\tilde{a}\in S$ with $a\ll \tilde{a}$. In particular, if $S$ contains a largest element (always happens when it is countably based or simple), usually denoted by $\infty$, then the condition of being stably finite is equivalent to $a\ll \infty$ (\cite[Paragraph 5.2.2.]{APT14}). Furthermore, if $S$ is simple, then $S$ is {\bf purely infinite} if $S=\{0,\infty\}$, i.e. $S$ only contains the zero and largest elements.

As shown in \cite[Proposition 5.2.10]{APT14}, the $\infty$ element in a simple $\Cu$-semigroup $S$ is not compact if and only if $S$ is stably finite. We will denote the elements way-below $\infty$ by
$$S_{\ll\infty}:=\{s\in S\mid s\ll\infty\}.$$
In fact, to ease notation, we will often denote it without $\infty$, i.e. $S_\ll$.

We finish this part about Cu-semigroups describing their functionals. By a functional on a $\Cu$-semigroup $S$ we mean a map $\lambda:S\to [0,\infty]$ that preserves addition, order, the zero element, and suprema of increasing sequences. Note that if $S$ is a simple $\Cu$-semigroup, then every functional is faithful (i.e. for nonzero $\lambda\in F(S),\, \lambda(x)\neq 0$ if $x\neq 0$). We denote the set of functionals on $S$ by $F(S)$. The differences and relations between functionals and states that one can associate to a $\Cu$-semigroup are deeply studied in \cite{BP18}. We also encourage the reader to look at \cite{Robert13} to know more about functionals in this framework. 

\subsection{Comparison properties}
Let us now recall the comparison conditions that form the heart of this article. The majority of the results stated here were proved in \cite{BP18}, and we encourage the reader to check it for further details.

We start this part studying $\omega$-comparison property for $\Cu$-semigroups. This property has been studied in several articles such as \cite{ BP18,NgCFP, OPR1, OPR2}, and it is implied for the majority of known regularity conditions of separable simple stably finite C*-algebras. Indeed, Ng coined this property as regularity, and it was in \cite[Corollary 4.2.5]{BRTTW} where the authors shown that regularity is equivalent to an algebraic condition on $\Cu(A)$ called $\omega$-comparison. The equivalence described in Definition \ref{defAlgw} is shown in \cite{BP18}.

\begin{definition}
	A C*-algebra $A$ is said to be regular if every full hereditary subalgebra of $A$, with no nonzero unital quotients and no nonzero bounded 2-quasitraces, is stable.
\end{definition}

\begin{definition}\label{defAlgw}
Let $S$ be a simple $\Cu$-semigroup. Then, we say $S$ satisfies $\omega$-comparison if any of the following equivalent conditions (and then all)  holds:
\begin{enumerate}
	\item Whenever $(y_n)$ is a sequence of nonzero elements in $S_{\ll\infty}$ such that $y_n<_sy_{n+1}$ for all $n$, then $\sum^\infty_{n=1}y_n=\infty$ (in $S$).
	\item Whenever $(y_n)$ is a sequence of nonzero elements in $S_{\ll\infty}$ such that $\lambda(\sum^\infty_{n=1}y_n)=\infty$ for all functionals $\lambda\in F(S)$, then $\sum^\infty_{n=1}y_n=\infty$.
\end{enumerate}
\end{definition}

We continue this part with the Corona Factorization Property (CFP for short). As stated in the introduction this was introduced by Elliott-Kucerovsky in  order to study when extensions of C*-algebras are absorbing (\cite{EK01}). This property has many equivalent definitions depending on the setting one works (see \cite{NgCFP} for further study). We recall below the characterization of this property found in \cite[Theorem 4.2]{NgCFP}. This determines structural results about stable C*-algebras.

\begin{definition}{\rm (\cite[Theorem 4.2]{NgCFP})}\label{def:AnCFP}
Let $A$ be a separable stable C*-algebra. It satisfies the Corona Factorization Property if given a full, hereditary subalgebra  $D$ of $A$ such that $M_n(D)$ is stable for some integer $n\geq 1$, then $D$ itself is stable.	
\end{definition}

Moving to the simple Cu-semigroup setting, this property was characterized in \cite{OPR1} as a certain comparison property associated to the Cuntz semigroup. Indeed, a $\sigma$-unital C*-algebra $A$ has the corona factorization property if and only if its Cuntz semigroup $\Cu(A)$ has the CFP as defined below (\cite{OPR2}). Let's us recall a couple of equivalent definitions of CFP in this setting.

\begin{definition}{\rm(\cite{BP18, RobertRordam})}\label{PropCFP}
	Let $S$ be a simple $\Cu$-semigroup. Then it satisfies CFP if any of the following equivalent properties  (and then all)  hold: 
	\begin{itemize}
		\item   Given any full sequence $(x_{n})_n$ in $S$, any sequence $(y_{n})_n$ in $S$, an element $x'$ in $S$ such that $x'\ll x_{1}$, and a positive integer $m$ satisfying $x_{n}\leq m y_{n}$ for all $n$, then there exists a positive integer $k$ such that $x'\leq y_{1}+\ldots+y_{k}$.
	 	\item Given a sequence $(y_n)_n$ in $S_{\ll\infty}$ such that $m\cdot \sum_{n=k}^\infty y_n=\infty$ for some $m$ and all $k\in\mathbb{N}$, then $\sum_{n=1}^\infty y_n=\infty$.
	\end{itemize}
\end{definition}

The last condition we want to recall is property {\rm (QQ)}. This was originally introduced in \cite{OPR2}, and it is deeply related to both conditions just defined (see \cite{BP18} for further details). As stated before, the behaviour of the largest element in a simple $\Cu$-semigroup is complicated to characterize, and the following property just concerns about this element.

\begin{definition}{\rm(\cite{OPR2})}
	A $\Cu$-semigroup $S$ satisfies the property {\rm (QQ)} if every element in $S$, for which a multiple is properly infinite, is itself properly infinite.  
\end{definition}

For simple $\Cu$-semigroups there is a unique properly infinite element, usually denoted by $\infty$. Hence, the above means that if $mx=\infty$ for some $m\in\mathbb N$, then $x=\infty$.

\section{stable elements}\label{Sec:Stable}
The notion of stable element can be found in \cite{KR00}, where the authors study finite, infinite and properly infinite positive elements in a C*-algebra. They show in \cite[Proposition 3.7]{KR00} that any stable element is properly infinite, but few is known about the converse. 
 We start recalling the definition of stable elements in $A$. 
\begin{definition}
	A positive element $a$ in a C*-algebra $A$ is called \emph{stable} if $\overline{aAa}$ is a stable C*-algebra.
\end{definition}
Along this section we deeply study the above notion, and we relate it to the next:
\begin{definition}
	Let $A$ be a separable simple C*-algebra. We say that $A$ is $\Cu$-stable whether for any representative $a\in (A\otimes \mathcal K)_+$ of $\infty\in \Cu(A)$, i.e. $\la a\ra=\infty$ in $\Cu(A)$, one has that $\overline{a(A\otimes \mathcal K)a}$ is a stable C*-algebra.
\end{definition}

Note that the above definition mainly concerns about the C*-algebra rather its Cuntz semigroup. In Theorem \ref{Thm:stableCSE} we show that this property can be characterized as a cancellation property for the Cuntz semigroup. Furthermore, $A$ being $\Cu$-stable implies that $A$ is stably finite (Theorem \ref{Thm:stableCSE}). Hence, we will often assume this fact whenever $\Cu$-stability is asked.
It is worth to point out that, equivalently, the above definition asks that any properly infinite element is stable. 

In order to study stable elements, we will use the following two continuous  $f_\delta,g_\delta:\mathbb R_+\to \mathbb R_+$ functions. We thank J. Gabe for suggesting this direction.

\vspace{-0.5cm}
\begin{multicols}{2}
	$$f_\delta(x) = \left\{
	\begin{array}{ll}
	t, & \text{ for }t\in[0,\delta/2] \\
	\delta-t, & \text{ for }t\in[\delta/2,\delta] \\
	0 & \text{ for }t\geq \delta
	\end{array}
	\right.$$
	
	$$\hspace{-0.5cm} g_\delta(x) = \left\{
	\begin{array}{ll}
	0, & \text{ for }t\in[0,\delta/2] \\
	1-\sqrt{\delta t^{-1}-1}, & \text{ for }t\in[\delta/2,\delta] \\
	1 & \text{ for }t\geq \delta
	\end{array}
	\right.$$
	\end{multicols}

Note that $f_\delta(t)=(1-g_\delta(t))t(1-g_\delta(t))$ and that $f_\delta(t)\perp g_{2\delta}(t)$ for all $\delta>0$. Moreover, $\{g_\delta(a)\}_{\delta>0}$ defines an approximate unit on $A_a=\overline{aAa}$, for any positive element $a\in A_+$.

\begin{lemma}\label{lem:Stable}
	Let $a\in A_+$ be a positive element. The following are equivalent:
	\begin{enumerate}[\rm (i)]
		\item $a$ is stable
		\item $f_\delta(a)$ is properly infinite and full in $\overline{aAa}$ for every $\delta>0$
		\item $f_\delta(a)$ Cuntz dominates every element in $\overline{aAa}$ for every $\delta>0$.
	\end{enumerate}
\end{lemma}
\begin{proof}
	{\rm (i) $\Rightarrow$ (ii)} Let $\delta >0$ be given. Then, $f_\delta(a)=(1-g_\delta(a))a(1-g_\delta(a))$ is stable (and thus properly infinite) by \cite[Corollary 4.3]{HR98} (notice separability is not needed in the proof, only $\sigma$-unitality). Suppose for contradiction that $f_\delta(a)$ is not full, and let $I$ be the two-sided closed ideal it generates. Then the spectrum of $a+I$ is contained in $[\delta,\|a\|]$, so $\overline{aAa}/I$ is unital. This contradicts stability of $a$, indeed every quotient of an stable C*-algebra is also stable, and thus $f_\delta(a)$ must be full for all $\delta>0$.
	
	{\rm (ii) $\Rightarrow$ (iii)} This follows from \cite[Proposition 3.5]{KR00}.
	
{\rm (iii) $\Rightarrow$ (i)}
Let $x\in \overline{aAa}$ and $\varepsilon>0$. Since $\{g_\delta(a)\}$ is an approximate identity for $\overline{aAa}$, pick $\delta>0$ such that $\|g_{2\delta}(a)xg_{2\delta}(a)-x\|<\varepsilon/2$, and define $x_0:=g_{2\delta}(a)xg_{2\delta}(a)$. By assumptions, $f_\delta(a)$ Cuntz-dominates $x_0$; hence, there exists $y$ such that $$y^*y=(x_0-\varepsilon/2)_+\,\,\text{ and }yy^*\in \overline{f_\delta(a)Af_\delta(a)}.$$
Recall that by construction, one has $f_\delta(a)\perp g_{2\delta}(a)$; therefore, $yy^*\perp y^*y$ with $\|y^*y- x\|<\varepsilon$. By the criteria described in \cite[Theorem 2.2]{HR98}, one obtains that $\overline{aAa}$ is stable as desired.
\end{proof}

After settling the background of stable elements, let us show the relation between the comparison properties exposed in Section \ref{Sec:Preliminaries} and stable elements. Let us start by the following result.

\begin{corollary}
	Let $A$ be a simple separable stably finite C*-algebra such that $\Cu(A)$ satisfies property $\mathrm{(QQ)}$. Then, $A$ is $\Cu$-stable.
\end{corollary}
\begin{proof}
	Let $a\in A\otimes\mathcal K$ be positive element such that $\la a\ra=\infty$ in $\Cu(A)$; namely, $a$ is a properly infinite element in $A\otimes\mathcal K$. Since $A$ is stably finite, it follows that $\infty\not\ll\infty$  by \cite[Proposition 5.2.10]{APT14}; therefore, $f_\delta(a)\neq 0$ for all $\delta >0$ (following notation of Lemma \ref{lem:Stable}).
	
	Using the fact that the Cuntz subequivalence is implemented by inclusion of supports in commutative C*-algebras, it follows that $a\precsim f_\delta(a)+ g_\delta(a)\precsim f_\delta(a)\oplus g_\delta(a)$ for all $\delta>0$. We can use simplicity of $\Cu(A)$ to find $n\in\mathbb N$ such that $\la g_\delta(a)\ra\leq (n-1)\la f_\delta(a)\ra$; then, $\la a\ra\leq n\la f_\delta(a)\ra$. Using property $\mathrm{(QQ)}$, one has that $\la f_\delta(a)\ra$ is equal to $\infty$ in $\Cu(A)$, and so $f_\delta(a)$ is properly infinite for all $\delta$. Therefore, $a$ is stable by Lemma \ref{lem:Stable}.
	\end{proof}
Notice that $\Cu$-stability of $A$ is a natural property to ask for any stably finite C*-algebra. Indeed, looking at $\Cu(A)$ from the Hilbert $A$-modules picture (see \cite{CEI08} for further details about this approach), $\Cu$-stability of $A$ asks for all Hilbert A-modules associated to different representatives of $\infty\in\Cu(A)$ to be isomorphic to $\ell^2(A\otimes\mathcal K)$. It is well-known that the Cuntz equivalence between Hilbert $A$-modules does not imply isomorphism between them (see \cite{BC09} for a concrete counterexample); however, that property holds under stable rank one assumption. Another extra property arising from stable rank one assumption is that $\Cu(A)$ becomes weak cancellative (as shown in \cite[Proposition 4.2]{RW10}). Next lemma uses the approximation displayed by R\o rdam-Winter to show that $\Cu$-stability of $A$ also gives us cancellation for big elements on $\Cu(A)$. 

	\begin{lemma}
	Let $A$ be a separable stably finite C*-algebra and $a,b,p\in (A\otimes\mathcal K)_+$ with $p$ a projection. Assume further that $b\oplus p$ is an stable element in $A\otimes\mathcal K$. If $$a\oplus p\precsim b\oplus p,$$  then $a\precsim b$. 
\end{lemma}
\begin{proof}
	Let us assume, without loss of generality, that $a,b$ and $p$ are orthogonal elements in $A\otimes\mathcal K$. By \cite[Corollary 2.56]{Th}, the Cuntz comparison for stable C*-algebras is unitarily implemented. Namely, there exists a unitary in the unitarization of $A\otimes\mathcal K$ such that $$u((a-\varepsilon)_++p)u^*\in \overline{(b+p)(A\otimes\mathcal K)(b+p)} =:B\otimes \mathcal K.$$
	Notice the latter equality holds due to the fact that $(b+p)$ is an stable element by assumptions.
	
	Now, $upu^*$ and $p$ are Murray-von Neumann equivalent in $B\otimes\mathcal K$; therefore they are also Peligrad-Szidó equivalent (see \cite{ORT}). Using \cite[Corollary 1.11]{BTZ18}, we extend the isometry defining the Peligrad-Szidó equivalence between $p$ and $upu^* $ to a unitary $v$ in $\mathcal M(B\otimes \mathcal K)$ satisfying that $vpv^*=upu^*$ in $B\otimes\mathcal K$.
	
	Now, we have that	
	$$v^*u(a-\varepsilon)_+u^*v\in B\otimes\mathcal K,\,\,v^*u(a-\varepsilon)_+u^*v\perp, v^*upu^*v=p, $$ 
	which provides that $v^*u(a-\varepsilon)_+u^*v$ belongs to $\overline{b(A\otimes\mathcal K)b}$. This shows that $(a-\varepsilon)_+\precsim b$; and as $\varepsilon>0$ was arbitrary, we conclude $a\precsim b$.
\end{proof}

A property of cancellation for big elements in $\Cu$-semigroups was already introduced in \cite{BP18}. Indeed, in the simple case, we say that a $\Cu$-semigroup $S$ satisfies Cancellation of Small elements at Infinity (CSE$\infty$ for short)  whether $x+y=\infty$ with $x\ll \infty$, implies that $y=\infty$. Next result shows the equivalence between the two notions under study: $\Cu$-stability of $A$ and CSE$\infty$. In particular, it characterizes cancellation of big elements in $\Cu(A)$. In order to show that, we need to ensure the existence of a projection in $A\otimes\mathcal K$. That is the reason why we assume the algebra to be unital.

\begin{theorem}\label{Thm:stableCSE}
	Let $A$ be a unital  separable simple C*-algebra. Then $A$ is $\Cu$-stable if and only if $\Cu(A)$ satisfies CSE$\infty$. In particular, $A$ is $\Cu$-stable implies that it is stably finite.
\end{theorem}
\begin{proof}
	For the "if" direction, let us consider a properly infinite element, denoted by  $a$, such that $f_\delta(a)\neq 0$ for all $\delta>0$. By construction it follows that $\infty=\la a\ra\leq \la f_\delta(a)\ra \oplus \la g_\delta(a)\ra$, with $\la g_\delta(a)\ra \ll \la a\ra$. By (CSE$\infty$), one has that $\la f_\delta(a)\ra$ is properly infinite for all $\delta>0$; hence, Lemma \ref{lem:Stable} implies that $a$ is stable as desired.
	
	For the converse direction, let us first check that $\Cu$-stability of $A$ implies it is stably finite. Then, we will show the desired implication in this case.
	
	 If $A$ was neither stably finite neither purely infinite, for any $x\in \Cu(A)$ there exists $n_x\in\mathbb N$ such that $n_xx=\infty$ (see \cite{BP18} for further details). Hence, considering the Cuntz class of the unit on $A$, we have that $n_{1}\la1_A\ra=\infty$ for some $n_1\in\mathbb N$. By $\Cu$-stability of $A$, one has that $1_{M_{n_1}}\otimes 1_A$ is a stable element. Hence, by Lemma \ref{lem:Stable}, $f_\delta(1_{M_{n_1}}\otimes 1_A)=1_{M_{n_1}}\otimes f_\delta(1_A)$ is properly infinite for all $\delta>0$. Since $1_A$ is a projection, there exists $\delta'>0$ such that $f_{\delta'}(1_A)=0$, what provides a contradiction since we requested  $f_\delta(1_A)\neq 0$ for all $\delta>0$.
	 
	  Now, assuming $A$ is stably finite, let us show that $\Cu$-stability of $A$ implies (CSE$\infty$). To this end, let $x+y=\infty$ with $x\ll\infty$. As before, assuming $A$ is unital, we may further assume that $x \leq n\la 1_A\ra$  for some $n$; hence, $x$ may be considered a compact element in $\Cu(A)$. Let a projection $p$ be one of its representatives (\cite{BC09}). This, indeed, can be done due $x\ll\infty$ and $\Cu(A)$ is simple.
	
	 Since any representative of the above Cuntz class is stable, denoting $x=\la p\ra$ and $y=\la b\ra$, one has that $p\oplus b$ is an stable element in $A\otimes \mathcal K$. Applying functional calculus to $p\oplus b$, one has that $f_\delta(p\oplus b)=f_\delta(p)\oplus f_\delta(b)$ for all $\delta>0$. Hence, by Lemma \ref{lem:Stable} $f_\delta(p)\oplus f_\delta(b)$ is properly infinite for all $\delta>0$. Since $p$ is a projection, there exists $\delta_0>0$ such that $f_{\delta_0}(p)=0$; therefore, $f_{\delta'}(p)\oplus f_{\delta'}(b)=f_{\delta'}(b)$ is properly infinite for all $\delta'<\delta_0$. That implies that $b$ is an stable element by Lemma \ref{lem:Stable}.	
	\end{proof}

As stated in \cite{BP18}, it is natural to wonder the following:
\begin{question}\label{q1}
	Does any separable simple and stably finite C*-algebra $A$ satisfy that it is $\Cu$-stable?
\end{question}
\begin{remark}
	Using the equivalence of both properties described in Theorem \ref{Thm:stableCSE}, one observes that Question \ref{q1} wonders about the uniqueness of orthogonal complement on  Hilbert $A$-modules. Indeed, looking at $\Cu(A)$ from the Hilbert $A$-modules picture described in \cite{CEI08}, $\Cu$-stability of $A$ says that whether $x+y=\infty$ in $\Cu(A)$ with $x\ll\infty$, then $y=\infty$. Namely, if $x=\la a\ra$ and $y=\la b\ra$, for $a,b\in (A\otimes\mathcal K)_+$, stability gives us: $$E_a\oplus E_b\sim \ell^2(A\otimes \mathcal K)\Rightarrow E_a\oplus E_b\cong \ell^2(A\otimes \mathcal K) \iff E_b\cong \ell^2(A\otimes \mathcal K).$$
	The above fact holds in the stable rank one setting. Indeed, in this framework the Cuntz equivalence is described by isomorphism of the associated Hilbert $A$-modules as shown in \cite{ORT}, and we have a weak cancellation by \cite{RW10}. Notice that $\Cu$-stability of $A$ asks for a "weak" version of stable rank one since we just need the above properties on the largest element of $\Cu(A)$. We express it in the next corollary.
\end{remark}
\begin{corollary}
Let $A$ be a  unital separable simple stably finite C*-algebra, and $a\in (A\otimes \mathcal K)_+$. Then, the following are equivalent:
\begin{itemize}
	\item orthogonal complement of $E_a$ on $\ell^2(A\otimes \mathcal K)$ is unique (up to isomorphism).
	\item Cuntz equivalence at $\infty\in \Cu(A)$ is induced by isomorphism of Hilbert $(A\otimes\mathcal K)$-modules.
	\item $A$ is $\Cu$-stable.
\end{itemize} 
\end{corollary}

If Question \ref{q1} had an affirmative answer, it would answer  \cite[Question 3.4]{KR00} in the setting under study, as we next state. 
\begin{corollary}
	Let $A$ be a unital separable simple $\Cu$-stable C*-algebra. Then, $\mathcal M(\overline{aAa})$ is properly infinite for any properly infinite element $a$ in $A$.
\end{corollary}
\begin{proof}
	By assumptions any properly infinite element $a\in A$ is automatically stable; namely $\overline{aAa}$ is stable. The multiplier algebra of a stable C*-algebra contains the bounded operators on an infinite-dimensional Hilbert space as a unital sub-C*-algebra; therefore, it is properly infinite
\end{proof}

We finish this section relating properly infiniteness with the notion of $(\omega,n)$-decompo\-sable as defined in \cite{RobertRordam}. Let us first recall this notion.
\begin{definition}\label{Def:wnDecomposable}
Let $A$ be a C*-algebra, $n\geq 1$ be an integer, and $u$ be an element in $\Cu(A)$. We say that $u$ is $(w,n)$-decomposable if there exists $x_1,x_2,\ldots$, different than zero, in $\Cu(A)$ such that $\sum^\infty_{i=1}x_i\leq u$ and $u\leq nx_i$ for all $i$.
\end{definition}

In \cite[Lemma 9.2]{RobertRordam}, the authors give several equivalent conditions to determine whether an element is $(\omega,n)$-decomposable. These different notions are used in \cite[Proposition 9.7]{RobertRordam} to provide a result relating stable C*-algebras and $(\omega,n)$-decomposability. However, the proof of the implication {\rm (i)}$\Rightarrow${\rm (ii)} in \cite[Proposition 9.7]{RobertRordam} is wrong. Indeed, they claim to find infinite mutually orthogonal positive elements in a non-stable C*-algebra, and this can not be done in general. Note that if \cite[Proposition 9.7]{RobertRordam} was true, it would imply that properly infinite elements are always stable (answering \cite[Question 3.4]{KR00} in the general setting). Assuming $\Cu$-stability for $A$, one get the following characterization of $(\omega,n)$-decomposable elements.

\begin{lemma}\label{Lem:xdecomposable}
	Let $A$ be a simple separable $\Cu$-stable C*-algebra, and $x=\la a\ra\in \Cu(A)$. Then,  
	\begin{enumerate}
			 \item If $x$ is $(\omega,n)$-decomposable, then $nx$ is properly infinite.
		\item If $nx$ is properly infinite for some $n\in\mathbb N$, then $x$ is $(\omega,n)$-decomposable.	\end{enumerate}

\end{lemma}
\begin{proof}
	The first part is trivial from the definition of $(\omega,n)$-decomposable. Indeed, by Definition \ref{Def:wnDecomposable} one has that $n\cdot\sum^\infty_{i=1}x_i\leq n\cdot x$ and $x\leq n\cdot x_i$; hence, $$\infty\cdot x\leq n\cdot\sum^\infty_{i=1}x_i=\sum_{i=1}^\infty n\cdot x_i\leq n\cdot x,$$ showing that $nx$ is properly infinite.
	
	For the second statement, let $x\in\Cu(A)$ such that $nx$ is properly infinite for some $n\in\mathbb N$, and consider the representative of $x$ given in the statement, i.e. a positive element $a\in A\otimes K$ such that $\la a\ra=x$. By $\Cu$-stability of $A$, it follows that $a\otimes 1_n$ is an stable element; therefore, $M_{n}((A\otimes K)_a)$ is stable. Use \cite[Lemma 5.3]{OPR1} to find a sequence $\{a_k\}$ of pairwise elements in $(A\otimes K)_a$ such that $\la (a-1/k)_+\ra\leq n\la(a-1/k)_+\ra\leq  n\la a_k\ra$ for all $k$ in $\Cu((A\otimes K)_a)$. Then, considering  $u=\la a\ra$, $x_k=\la a_k\ra$ and $y_k=\la(a-1/k)_+\ra$, it follows from \cite[Lemma 9.2(iii)]{RobertRordam} that $\la a\ra$ is $(\omega, n)$-decomposable as desired.
		\end{proof}

\begin{corollary}\label{cor:CFP=QQ}
Let $A$ be a simple separable  $\Cu$-stable C*-algebra. Then, $A$ satisfies the Corona Factorization property if and only if $\Cu(A)$ satisfies the property ${\rm (QQ)}$.
\end{corollary}
\begin{proof}
	The "only if" part is already proven in \cite{BP18}, so let us show the converse.  Let $x\in \Cu(A)$ such that $nx=\infty$ for some $n$. Then, one has that $x$ is $(\omega,n)$-decomposable by Lemma \ref{Lem:xdecomposable}. Using the characterization of Corona Factorization Property described in \cite[Proposition 9.3]{RobertRordam}, one has that any $(\omega,n)$-decomposable element is properly infinite; therefore, the desired result follows.
\end{proof}

	Focusing on the real rank zero setting, recall that if $A$ is simple and neither stably finite nor purely infinite case, then we have that any element in $\Cu(A)$ is $(\omega,n)$-decomposable for some $n$  by \cite[Proposition 4.2]{OPR1}. Using this fact, following the lines of the above proof, we obtain the well-known dichotomy between stably finite and purely infinite under the assumptions of simplicity, real rank zero and Corona Factorization property (Zhang's Theorem).

\section{Projections on $\mathcal M(A\otimes\mathcal K)$ and Asymptotic $\omega$-comparison}\label{sec:asymomega}

In this section we seek to study asymptotic regularity for separable simple and stably finite C*-algebras, as defined by Ng in \cite{NgCFP}. The final goal is to show the equivalence between this property and property $R$ (Theorem \ref{Thm:RAssym}).  
 This section has been build on from a collaboration effort started some years ago with M. Christensen. In particular, some parts of this section already appeared in M. Christensen's PhD-thesis.

As happened with both $\omega$-comparison and the Corona Factorization property in \cite{OPR2}, we start rewriting asymptotic regularity condition in terms of an order property associated to the Cuntz semigroup. The equivalence between these properties is exposed in Lemma \ref{lem:AsComp}. After rephrasing this condition, we immerse into the $\mathcal M(A\otimes\mathcal K)$ framework to show the desired equivalence between property $R$ and asymptotic regularity (Theorem \ref{Thm:RAssym}). 

Let's start recalling that a C*-algebra $D$ is said to have property {\rm (S)} if it has no unital quotients and admits no bounded 2-quasitraces. This property tries to characterize stable C*-algebras as explained in the introduction, and it is the milestone of the current section. An equivalent definition of property {\rm (S)} is given in \cite[Proposition 4.5]{OPR2}. This states that $D$ satisfies property {\rm (S)} if and only if for all $a\in F(D)_+$, there exists $b\in D_+$ such that $ab=0$ and $\la a\ra<_s\la b\ra$ in $\Cu(D)$, where $F(D):=\{a\in A_+\mid ae=e \text{ for some }e\in A_+ \}$.  We provide an equivalent condition in Proposition \ref{prop:PropS} using the projections of $\mathcal M(A\otimes\mathcal K)$.
With this in mind, we recall the next definition. 

\begin{definition}[Asymptotically regular]
	Given a separable C*-algebra $A$, we say that $A$ is asymptotically regular if, for any full hereditary C*-subalgebra $D$ of $A\otimes K$ with property {\rm (S)}, there exists an integer $n\in\mathbb N$ such that $M_n(D)$ is stable.
\end{definition}

In order to determine the above property, we will relate it to the following condition in the Cuntz semigroup of $A$. Due to its relationship with $\omega$-comparison, it is natural to call it asymptotic $\omega$-comparison.

\begin{definition}
	Let $S$ be a simple $\Cu$-semigroup. We say that $S$ satisfies asymptotic $\omega$-comparison if the following holds:
	\begin{itemize}
		\item  whenever $y_1,y_2,\ldots$ is a sequence of non-zero elements in $S_\ll$, such that $y_i<_sy_{i+1}$ for all $i\geq 1$, there exists $n\in\mathbb N$ such that $n\sum^\infty_{i=m}y_i=\infty$ for all $m\geq 1$.
	\end{itemize}
\end{definition}

In a similar fashion that it is shown an equivalent definition of $\omega$-comparison via the use of functionals of $\Cu(A)$ in \cite{BP18}, we are able to conclude the following proposition. We omit the proof, and recommend to look at \cite{BP18} for further details.
\begin{proposition}\label{Prop:asymw}
	Let $S$ be a simple $\Cu$-semigroup. Then the following are equivalent:
	\begin{enumerate}
		\item $S$ has asymptotic $\omega$-comparison.
		\item Whenever $y_1,y_2,\ldots$ is a sequence of non-zero elements in $S_\ll$ satisfying the condition that $\lambda(\sum^\infty_{i=1}y_i)=\infty$ for all non-zero functionals $\lambda$ on $S$, there exists $n\in\mathbb N$ such that $n\sum^\infty_{i=m}y_i=\infty$, for all $m\in\mathbb N$.
	\end{enumerate}
\end{proposition}
Let us start with an easy lemma, which rewrites the first part of Proposition \ref{PropCFP} {\rm (iv)}.

\begin{lemma}\label{lem:RewriteCondCFP}
	Let $S$ be a Cu-semigroup and $y_1,y_2,\ldots$ be any sequence of elements in $S$. Then the next holds:
	\begin{enumerate}
		\item If there is an $n\in\mathbb N$ such that $n\sum^\infty_{i=m}y_i=\infty$ for all $m\in \mathbb N$, then, for every $j\geq 1$, it holds that $\sum^j_{k=1}y_k\leq n\sum^\infty_{i=j+1}y_i.$
		\item If there is an $n\in\mathbb N$ such that $n\sum^\infty_{i=1}y_i=\infty$ and $n\sum^{j}_{k=1}y_k\leq n\sum^\infty_{l=j+1}y_l$ for all $j\geq 1$, then $2n\sum^\infty_{l=m}y_l=\infty$ for all $m\geq 1$.
	\end{enumerate}
\end{lemma}

\begin{proof}
	The first statement is obvious, so let's show the second. To this end, let $y_1,y_2,\ldots$ a sequence in $S$ and $n\in\mathbb N$ as described in the statement.
	Then, for an arbitrary $m\in\mathbb N$ it follows 
	$$2n\sum^\infty_{l=m}y_l=n\sum^\infty_{l=m}y_l+n\sum^\infty_{l=m}y_l\geq n\sum^\infty_{i=1}y_i=\infty,$$ getting the desired inequality.
\end{proof}

In order to conclude the desired equivalence we need the following lemmas. The first is a mild elaboration of \cite[Lemma 4.3]{OPR2} and needs the next definition.

Given a C*-algebra $D$, and a strictly positive contraction $c\in D_+$, let $$L_c(D):=\{a\in D_+\mid g_\varepsilon(c)a=a\text{ for some }\varepsilon>0\},$$ where $g_\varepsilon$ is the continuous function defined before  Lemma \ref{lem:Stable}.

Note that, if $a,b\in L_c(D)$, then $a+b\in L_c(D)$ and $g_\delta(c)dg_\delta(c)\in L_c(D)$, for every $\delta>0$ and $d\in D_+$. 

\begin{lemma}\label{lem:LcD}
	Let $D$ be a $\sigma$-unital C*-algebra with property {\rm (S)}, and let $c\in D$ be a strictly positive contraction. Then, for every $a\in L_c(D)$, there exists $b\in D_+$ such that $ab=0$, $\la a\ra<_s\la b\ra,$ and $b\in L_c(D)$.
\end{lemma}
\begin{proof}
	Choose $\varepsilon'>0$ such that $g_{\varepsilon'}(c)a=a$ and denote by $e:=g_{\varepsilon'}(c)$. Note that $a\precsim (e-1/2)_+$. Since $e\in L_c(D)$, and $D$ has property {\rm (S)}, there exists $b_0\in D_+$ such that $eb_0=0$ and $\la e\ra<_s\la b_0\ra$ by \cite[Proposition 4.5]{OPR2}. Moreover, there exists $\delta>0$ such that $\la(e-1/2)_+\ra<_s\la(b_0-\delta)_+\ra$.  
	
	Since $\{g_{1/m}(c)\}$ is an approximate unit for $D$, we may choose $m\in \mathbb N$ such that $\varepsilon:=1/m<\varepsilon'/2$ and $\|b_0-g_{1/m}(c)b_0g_{1/m}(c)\|<\delta$. Moreover, the element $g_{1/m}(c)b_0g_{1/m}(c)$ belongs to $L_c(D)$ and it is orthogonal to $a$. Hence, by \cite[Lemma 2.5]{KR00} one has $$\la a\ra\leq \la (e-1/2)_+\ra<_s\la (b_0-\delta)_+\ra\leq \la g_{1/m}(c)b_0g_{1/m}(c)\ra,$$ as desired.
\end{proof}

We are now in the position of showing the equivalence between both notions.
\begin{lemma}\label{lem:AsComp}
	Let $A$ be a simple and separable C*-algebra. Then the following are equivalent:
	\begin{enumerate}
		\item $\Cu(A)$ has asymptotic $\omega$-comparison.
		\item $A$ is asymptotically regular
	\end{enumerate}
\end{lemma}
\begin{proof}
	Assume $\Cu(A)$ has asymptotic $\omega$-comparison, and let $D\subseteq A\otimes\mathcal K$ be a non-zero hereditary sub-C*-algebra, with property {\rm (S)}. Let $c$ be a strictly positive element in $D$. Then, for every $m\geq 1$, the element $c\otimes 1_m\in D\otimes M_m\cong M_m(D)$ is strictly positive. Hence, proving the stability of $M_m(D)$, it suffices to prove the existence of $n\in\mathbb N$ such that, for all $\varepsilon>0$, there exists $b\in (D\otimes M_m)_+$ satisfying that $(c\otimes 1_m-\varepsilon)_+\perp b$, and $(c\otimes 1_m-\varepsilon)_+\precsim b$ by \cite[Theorem 2.1]{HR98}.
	
	Let $(\varepsilon_n)_{n\geq 1}$ be a decreasing sequence of positive real numbers such that $\varepsilon\to 0$. We prove, by induction, that there exists a sequence $(b_n)$ of pairwise orthogonal, positive elements in $D$, such that $\la b_i\ra<_s\la b_{i+1}\ra$, $b_i\in L_c(D)$, $\la (c-\varepsilon_i)_+\ra<_s\la b_i\ra,$ and $(c-\varepsilon_i)_+\perp b_i$ for all $i\geq 1$. The induction starts from Lemma \ref{lem:LcD}. Now, let $b_1,\ldots,b_n$ satisfying the desired properties. By the properties of $L_c(D)$, we have that $$(c-\varepsilon_{n+1})_++b_1+\ldots+b_n\in L_c(D),$$ whence, applying Lemma \ref{lem:LcD}, it follows that there exists $b_{n+1}$ satisfying the desired properties.
	
	For $i\geq 1$, let $y_i:=\la b_i\ra\in \Cu(A)$. Now, by asymptotic $\omega$-comparison it follows that there exists $n\in\mathbb N$ such that $n\sum^\infty_{i=m}y_i=\infty$ for all $m\in \mathbb{N}$. Let $\varepsilon>0$ be arbitrary and choose $m\in\mathbb{N}$ such that $\varepsilon_m<\varepsilon$. Then, there exists $N\in\mathbb N$ such that $$\la(c\otimes 1_n-\varepsilon_m)_+\ra\leq n\sum^N_{i=m}\la b_i\ra=\la \sum^N_{i=m}b_i\otimes 1_n\ra.$$ Since $(c\otimes 1_n-\varepsilon)\leq (c\otimes 1_n-\varepsilon_m)\perp \sum^N_{i=m}b_i\otimes 1_n$, the desired result follows.
	
	For the converse, let $y_1,y_2,\ldots$ a sequence of non-zero elements in $\Cu(A)$ such that $y_i<_sy_{i+1}$ for all $i\geq 1$. Choose pairwise orthogonal positive elements $b_i\in A\otimes K$ such that $\|b_i\|\leq 2^{-i}$ and $y_i=\la b_i\ra$. Set $b:=\sum^\infty_{i=1}b_i$ and let $D\subseteq A\otimes \mathcal K$ denote the hereditary sub-C*-algebra generated by $b$. We show that $D$ has property {\rm (S)}. It is easy to see that $\lambda(\la b\ra)=\infty$ for all functionals $\lambda$ on $\Cu(A)$; hence, $D$ does not admit any bounded 2-quasitrace. Similarly, assuming for a contradiction that $D$ is unital, it follows that $b$ is invertible, and therefore $\sum^m_{i=1}b_i$ is invertible for some $m\in\mathbb N$. Therefore, $b_k=0$ for all $k>m$ since these elements are orthogonal to an invertible element. This implies that $b_i=0$ for all $i$ due $\la b_i\ra<_s\la b_j\ra$ whenever $i<j$, a contradiction.
	
	Therefore, $D$ has property {\rm (S)} and $D\otimes M_n$ is stable for some $n\in\mathbb N$ by asymptotic regularity of $A$. Now,  by Lemma \ref{lem:RewriteCondCFP} one needs to show  $$\sum^j_{k=1}\la (b_k\otimes 1_n-\varepsilon)_+\ra\leq \sum^\infty_{i=j+1}\la b_i\otimes 1_n\ra \text{ for every }j\geq 1 \text{ and }\varepsilon>0$$ to conclude the Corona Factorization Property.
	
	To this end, fixing $\varepsilon>0$, and denoting $d_i:=b_i\otimes1_n$, for all $i\geq 1$ build
	
	\begin{equation}
	e_m^j:=\sum^m_i=j g_{1/m}(d_i)\text{ and } e_m:=e_m^1 \text{ for every }j\geq 1\,\,\, m\geq 1.
		\end{equation}
	By functional calculus, it follows that $(e_m)_{m\geq 1}$ is an approximate unit for $D\otimes M_n$. Moreover, $e_m^j\precsim \sum^m_{i=j}d_i$ for all $j,m\in\mathbb N$. 
	
	Let us now use that $D\otimes M_n$ is stable. Then, there exists $a\in D\otimes M_n$ such that $\sum^j_{k=1}d_k\perp a$ and $\sum^j_{k=1}d_k\precsim a$. We may therefore choose $\delta>0$ such that $$\sum^j_{k=1}(d_k-\varepsilon)_+\precsim (a- \delta)_+.$$ 
	By construction, we can choose $m\in\mathbb N$ such that $\|a-e_mae_m\|<\delta$; therefore, the above orthogonality implies $$e_ma^{1/2}=(\sum^m_{i=1}g_{1/m}(d_i))a^{1/2}=(\sum^m_{i=j+1}g_{1/m}(d_i))a^{1/2}=e_m^{j+1}a^{1/2}.$$
	
	In particular, $e_mae_m\leq \|a\|(e_m^{j+1})^2$;hence,
	$$\sum^j_{k=1}(d_k-\varepsilon)_+\precsim (a-\delta)+\precsim e_mae_m\precsim (e^{j+1}_m)^2\precsim \sum^m_{i=j+1}d_i,$$
	as desired.
\end{proof}

The above equivalence shows that asymptotically regular implies dichotomy in our setting. 

\begin{proposition}\label{Prop:AsComp}
	Let $A$ be a simple and separable C*-algebra. Then it is either stably finite or purely infinite if it is asymptotically regular.
\end{proposition}
\begin{proof}
	Assuming $A$ is neither stably finite nor purely infinite, then  every $z\in \Cu(A)$ satisfies $nz=\infty$ and there exists $x\in\Cu(A)$ such that $x\ll\infty$ and $x\neq\infty$. Using Glimm's halving property (i.e. for each $y\in\Cu(A)$ there exists $z\in\Cu(A)$ such that $2z\leq y$) we may find a sequence $x_0:=x,x_1,x_2,\ldots\in\Cu(A)$ such that $2^{i+1}x_i\leq x_0$ for all $i\geq 1$. Moreover, by induction it follows that $\sum^\infty_{i=j}x_i\leq 2x_j$ for all $j\geq 1$.
	
	In particular, the latter implies that $2^j\sum^\infty_{i=j}x_i\leq 2^{j+1}x_j\leq x_0$ a finite element. Since every element in $\Cu(A)$ is eventually infinite, it follows that $x_i<_sx_{i+1}$ for all $i$, but there exists no $n\in\mathbb N$ such that $n\sum^\infty_{i=m}x_i=\infty$ for all $m\in\mathbb N$. Namely, $A$ cannot be asymptotically regular.
\end{proof}

We finish this section showing the equivalence between asymptotic regularity and property $R$, as defined in \cite{NgCFP}. Let us first introduce some notation and background in order to understand better the result.

As explained in section \ref{Sec:Preliminaries}, the Cuntz semigroup of a separable C*-algebra $A$ can be described via different frameworks. A concrete characterization between hereditary subalgebras of $A\otimes\mathcal K$ and projections in $\mathcal M(A\otimes\mathcal K)$ was explicitly written by Kucerovsky in \cite[Lemma 10]{Ku04}:

\begin{lemma}{\rm(\cite{Ku04})}\label{lemKucer}
	Let $A\otimes\mathcal K$ be a separable C*-algebra. Then, for every hereditary subalgebra $ B\subset A\otimes \mathcal K$, there exists a multiplier projection $P\in \mathcal M(A\otimes\mathcal K)$ such that $P(A\otimes\mathcal K)P\cong  B$.
\end{lemma}

This approach is very useful to determine when a full hereditary subalgebra of $A\otimes\mathcal K$ is stable. Indeed, we have the next important result shown by Brown in \cite[Theorem 4.23]{Bro88}.
\begin{theorem}{\rm(\cite{Bro88})}\label{ThmBrown}
	Let $A\otimes\mathcal K$ be a separable C*-algebra, and $P$ a multiplier projection in $\mathcal M(A\otimes\mathcal K)$. Then $P(A\otimes\mathcal K)P$ is a stable, full, hereditary subalgebra of $A\otimes\mathcal K$ if and only if $P$ is Murray-von Neumann equivalent to the unit of $\mathcal M(A\otimes\mathcal K)$.
\end{theorem}

This analytical side provides us a new definition of property (S), as stated below. We start recalling some definitions.
\begin{definition}
	Let $A$ be a unital C*-algebra. We say that $A\otimes\mathcal K$ has property R, if whenever $p$ is a projection contained inside a proper ideal of $\mathcal M (A\otimes\mathcal K)$, then $p$ is also contained inside a regular ideal.
\end{definition}
Recall that any unital trace on $A$ extends canonically to a trace on the positive cone of $\mathcal M(A\otimes\mathcal K)$. We use this extension to say that a proper ideal $\mathcal J$ of $\mathcal M(A\otimes\mathcal K)$ is regular if $\mathcal J$ is contained in an ideal of the form $\mathcal J_\tau$ (norm-closure of $\{b\in \mathcal M(A\otimes\mathcal K)\mid \tau(b^*b)<\infty\}$), for some unital $\tau$ on $A$. Otherwise, we say that $\mathcal J$ is non-regular.

\begin{proposition}\label{prop:PropS}
	Let $A$ be a unital separable simple exact C*-algebra, and $P$ a multiplier projection  in $\mathcal M(A\otimes\mathcal K)$ defining a full hereditary C*-subalgebra $P(A\otimes\mathcal K)P$ of $  A\otimes \mathcal K$. Then, $P(A\otimes\mathcal K)P$ satisfies property (S) if and only if $P$ belongs to a non-regular ideal of $\mathcal M(A\otimes\mathcal K)$.
\end{proposition} 
\begin{proof}
	Let $P$ be a multiplier projection in $\mathcal M(A\otimes\mathcal K)$ as in the statement. Then, assume that the hereditary C*-subalgebra $P(A\otimes\mathcal K)P$ satisfies property (S). The definition of property (S) asks for the lack of bounded 2-quasitraces; hence, $P$ cannot be contained in any regular ideal of $A\otimes\mathcal K$. Namely, $P$ belongs to a non-regular ideal.
	
	For the converse, let us assume that $P$ belongs to a non-regular ideal of $\mathcal M(A\otimes\mathcal K)$. Then, the hereditary C*-subalgebra in $A\otimes\mathcal K$ defined via $P$, i.e. $P(A\otimes\mathcal K)P\subseteq A\otimes\mathcal K$, has  no nonzero bounded 2-quasitraces and no quotient is unital due simplicity of $A$. Therefore, $P(A\otimes\mathcal K)P$ satisfies property (S) by definition.
	\end{proof}

Using Brown's result described in Theorem \ref{ThmBrown}, we use the above Proposition to see when property (S) implies stability (up to some matrix extension). Indeed, this happens when there is no proper non-regular ideals in $\mathcal M(A\otimes\mathcal K)$. In other words, when property R holds. Next result shows the equivalence between this property and asymptotic regularity. 

\begin{theorem}\label{Thm:RAssym}
	Let $A$ be a unital separable simple exact C*-algebra. Then $A\otimes\mathcal K$ is asymptotically regular if and only if $A\otimes \mathcal K$ satisfies property $R$.
	\end{theorem}

\begin{proof}
	Let us first show the "only-if" direction. Let $P$ be a multiplier projection in $\mathcal M(A\otimes \mathcal K)$ such that $P$ is contained in a proper ideal of $\mathcal M(A\otimes\mathcal K)$. If $P\in A\otimes \mathcal K$, then automatically $P$ is contained in a regular ideal of $\mathcal M(A\otimes \mathcal K)$; hence, we assume that $P$ is not in $A\otimes\mathcal K$.
	
	In this case, suppose that $P$ is not contained in any regular ideal of $\mathcal M(A\otimes \mathcal K)$. This implies that $P(A\otimes\mathcal K)P$ must be a full hereditary subalgebra of $A\otimes\mathcal K$, with no unital quotient and no nonzero bounded 2 quasitraces, i.e., it satisfies property (S). By the assumption of asymptotic regularity, there exists $n\in\mathbb N$ such that $M_n(P(A\otimes\mathcal K)P)$ is stable; namely, the sum of $n$-copies of $P$, i.e. $\oplus^nP$, is Murray-von Neumann equivalent to the unit of $\mathcal M(A\otimes\mathcal K)$ (see Theorem \ref{ThmBrown}). This contradicts our assumption that $P$ is contained in a proper ideal of $\mathcal M(A\otimes \mathcal K)$; therefore, $P$ belongs to a regular ideal of $\mathcal M(A\otimes\mathcal K)$, and property $R$ holds.
	
	 For the converse, let $D$ be a nonzero hereditary subalgebra of $A\otimes\mathcal K$ satisfying property (S). Use Lemma \ref{lemKucer} to find the multiplier projection $P$ in $\mathcal M(A\otimes \mathcal K)$ such that $P(A\otimes\mathcal K)P$ is isomorphic to $D$. By the lack of nonzero bounded 2-quasitraces given by property (S), the projection $P$ cannot be contained in a regular ideal of $\mathcal M(A\otimes\mathcal K)$. Therefore, $P$ is a norm-full element in $\mathcal M(A\otimes\mathcal K)$ due to property $R$.
	 
	  Since $P$ is norm-full element, there exists $n\in\mathbb N$ such that $\oplus^n P$ is Murray-von Neumann equivalent to the unit of $\mathcal M(A\otimes\mathcal K)$. Therefore, $M_n(P(A\otimes\mathcal K)P)\cong M_n(D)$ is stable by Theorem \ref{ThmBrown} as desired.
	\end{proof} 

A simple combination of Proposition \ref{Prop:AsComp} and Theorem \ref{Thm:RAssym} provides the following:
\begin{corollary}\label{Cor:Dich}
Let $A$ be a unital separable simple exact C*-algebra such that $A\otimes \mathcal K$ satisfies property $R$. Then $A$ is either stably finite or purely infinite.
\end{corollary}

\begin{remark}
	One may wonder whether the converse of last corollary (or equivalently Proposition \ref{Prop:AsComp}) is true. The answer to that is negative, i.e. there exist a separable simple stably finite C*-algebra that does not satisfy property $R$. The counterexample satisfying that condition appeared in \cite{Petz13}. In there it is constructed a stably finite projection $Q$ in the multiplier algebra of a separable stable simple C*-algebra $A$, which satisfies that $Q=\sum^\infty_{j=1}p_j,$ where $p_j$ are pairwise orthogonal projections in $A$ satisfying $\lambda(p_i)=\lambda(p_j)<\infty$ for all $i, j$ and all functionals $\lambda$.
	
	 Stably finiteness is not specified in \cite{Petz13}, but it follows from the  definition of projection $Q$. Indeed, if for the contrary $\infty\in \Cu(A)$ was a compact element, then some multiple of $\la \sum^\infty_{i=1}p_j\ra$ would be equal than $\infty$. Using compacity of $\infty$, that implies that $n\sum^m_{i=1}p_i=\infty$ for some $n,m\in\mathbb N$. Namely, $\lambda(p_i)=\infty$ for all functionals, contradicting the finiteness of them. 
	 
	 This example shows that not all stably finite exact C*-algebras admits a bounded trace.
\end{remark}

An equivalent formulation of Corona Factorization Property for separable simple C*-algebras is that any norm-full multiplier projection $P\in\mathcal M(A\otimes K)$ is Murray-von Neumann equivalent to $1_{\mathcal M(A\otimes \mathcal K)}$ (see \cite{KN06} for further details). Roughly speaking, {\rm CFP} allows us to pass from $n$-size of stability, to stability itself. Hence, Theorem \ref{Thm:RAssym} induces the following:
\begin{corollary}\label{Cor:Equal}
	Let $A$ be a unital separable simple exact C*-algebra. Then, the following are equivalent:
	\begin{enumerate}
		\item $A$ is regular 
		\item $A$ is asymptotically regular and satisfies {\rm CFP}
		\item $A$ satisfies property $R$ and  {\rm CFP}
	\end{enumerate}
\end{corollary}

Focusing on the real rank zero setting, Ng also described asymptotic regularity in \cite[Proposition 4.9]{NgCFP}. In particular, he showed the following, which we recall (without proof) by completeness.
\begin{proposition}\label{prop:notstable}
	Let $A$ be separable simple exact real rank zero C*-algebra. If there exists a norm-full multiplier projection $P\in\mathcal M(A\otimes\mathcal K)$ not Murray-von Neumann equivalent to $1_{\mathcal M(A\otimes\mathcal K)}$, then there exists a full hereditary subalgebra $D\subseteq A\otimes\mathcal K$ satisfying property (S) such that $M_n(D)$ is not stable for any $n\in\mathbb N$.
\end{proposition}
 Above proposition easily provides the argument that, in the real rank zero framework, asymptotic regularity (or property R) implies Corona Factorization Property. Indeed, if there was a norm-full projection different than the unit in $\mathcal M(A\otimes\mathcal K)$, then asymptotic regularity would not hold. Therefore, by Corollary \ref{Cor:Equal} it follows that Regularity, Asymptotic regularity and Property R are equivalent in this setting (as shown in \cite[Theorem 4.14]{NgCFP}). 
 
  By Proposition \ref{prop:PropS}, one has the latter condition exposed in Proposition \ref{prop:notstable} is
  equivalent to the existence of a projection $P$ in a proper non-regular ideal in $\mathcal M(A\otimes\mathcal K)$. 
  In this context, to build an example satisfying {\rm CFP} and not being regular, one needs to build an algebra $A$ such that both $\mathcal M(A\otimes \mathcal K)$ has a projection in a proper non-regular ideal and all norm-full multiplier projections are Murray-von Neumann equivalent to $1_{\mathcal M(A\otimes\mathcal K)}$. We recommend to look at \cite{KN07} to learn more about non-regular ideals in the multiplier algebra of an stable C*-algebra.
  
  	\ \newpage
  {\bf Acknowledgments}. The author would like to thank J. Gabe for suggesting Lemma \ref{lem:Stable} and M. Christensen for preliminary discussions about asymptotic regularity. Moreover, I would also like to thank P. Ara for their comments on an earlier version of this paper.


\begin{thebibliography}{9}

	\bibitem{APT} P. Ara, F. Perera, A. S. Toms; K-theory for operator algebras , classification of C*-algebras, Contemp. Math. 534, 'Aspects of operator algebras and applications', Amer. Math. Soc. (2011), 1--72.
	\bibitem{APT14} R. Antoine, F. Perera, H. Thiel; Tensor products and regularity properties of Cuntz semigroups, Memoirs of the Amer. Math. Soc., 251 (1199), pp. 199 pp., 2018.
	\bibitem{BRTTW} B. Blackadar, L. Robert, A. Toms, A. Tikuisis, W. Winter; An algebraic approach of the radius of comparison, Trans. Amer. Math. Soc. 364:7 (2012), 3657--3674.

\bibitem{BC18} J. Bosa, M. Christensen; The failure of the Corona Factorization Property for the Villadsen Algebra $\mathcal V_\infty$, Math. Scand., 123-1, (2018), pp 142-146.
	\bibitem{BP18} J. Bosa, H. Petzka; Comparison properties of the Cuntz semigroup and applications to C*-algebras, Canadian J. of Mathematics, 70 (2018), no. 1, pp 26-52.
	\bibitem{BTZ18} J. Bosa, G. Tornetta, J. Zacharias; Open projections and suprema in the Cuntz semigroup, Mathematical Proceedings Cambridge Philosophical Society 164 (2018), no. 1, 135-146. 
\bibitem{Bro88} L. Brown; Semicontinuity and multipliers of C*-algebras, Canadian J. of Mathematics, XL (1988) no. 4, 865-988
	\bibitem{BC09} N. Brown, A. Ciuperca; Isomorphism of Hilbert modules over stably finite
	C*-algebras, J. of Funct. Anal. 257 (2009) 332–339
	
	\bibitem{ChTh} M. Christensen; Regularity of C*-algebras and Central Sequence Algebras, PhD Thesis, University of Copenhagen, 2017. 
	\bibitem{Cu78} J. Cuntz; Dimension functions on simple C*-algebras, Math. Ann. 233 (1978),145--153.
	\bibitem{CEI08} K. T. Coward, G. A. Elliott, C. Ivanescu; The Cuntz semigroup as an invariant for C*-algebras, J. Reine. Angew. Math. 623 (2008), 161--193.

	\bibitem{EK01} G. A. Elliott, D. Kucerovsky; An abstract Voiculescu-Brown-Douglas-Fillmore absorption theorem, Pacific J. Math. 198:2 (2001), 385--409.
	\bibitem{HRW07} I. Hirshberg, M. R\o rdam, W. Winter; $C_0(X)$-algebras, stability and strongly self-	absorbing C*-algebras, Math. Ann. 339 (2007), no. 3, 695–732.
	\bibitem{HR98} J. v. B. Hjelmborg, M. R\o rdam; On stability of C*-algebras, J. Funct. Anal., 155(1):153-170, 1998.
	\bibitem{KR00} E. Kirchberg and M. R\o rdam; Non-simple purely infinite C*-algebras. Amer. J. Math., 122(3):637–666, 2000.
	\bibitem{Ku04}D. Kucerovsky; Extensions contained in ideals Trans. of the AMS, 356 (2004) no. 3, 1025-1043 
	\bibitem{KN06} D. Kucerovsky, P. W. Ng; S-regularity and corona factorization property, Math. Scand. 99:2 (2006), 204--216.
	\bibitem{KN07} D. Kucerovsky, P. W. Ng; Non-regular ideals in the Multiplier algebra of a stable C*-algebra, Houston J. of Math., 33, no.4 (2007), 1117-1130.
	\bibitem{NgCFP}P. W. Ng; The corona factorization property, in "Operator Theory, Operator Algebras and Applications", Contemp Math. 414 Amer. Math. Soc. (2006), 97--111.   
	\bibitem {OPR1} E. Ortega, F. Perera, M. R\o rdam; The corona factorization property and refinement monoids, Trans. Amer. Math. Soc. 363:9 (2011), 4505--4525. 
	\bibitem{OPR2} E. Ortega, F. Perera, M. R\o rdam; The corona factorization property, stability, and the Cuntz semigroup of a C*-Algebra, Int. Math. Res. Not. 1 (2012), 34--66. 
	\bibitem{ORT} E. Ortega, M. R\o rdam, H. Thiel; The Cuntz semigroup and comparison of open projections. J. Funct. Anal. 260 (2011), 3474-3493. 
	\bibitem{Petz13} H. Petzka; The Blackadar–Handelman theorem for non-unital C*-algebras, J. Funct. Anal.  264 (2013), 7, 1547-1564
	\bibitem{Ro97} M. R\o rdam; Stability of C*-algebras is not an stable property,  Doc. Math. J. DMV , 2, (1997), 375-386. 
	\bibitem{RoSt} M. R\o rdam; Stable C*-algebras, Adv. Studies in Pure Mathematics 38 "Operator Algebras and Applications". Edited by Hideki Kosaki. (2004), 177-199.
	\bibitem{RW10} M. R\o rdam, W. Winter; The Jiang-Su algebra revisited. J. Reine Angew. Math. 642 (2010), 129-155. 
	\bibitem{Robert13} L. Robert; The cone of functionals on the Cuntz semigroup, Math. Scand., 113:2 (2013), 161--186.
	\bibitem{RobertRordam} L. Robert, M. R\o rdam; Divisibility properties for C*-algebras, Proc. London Math. Soc. 106:6 (2013), 1330--1370 
	\bibitem{Th} H. Thiel; The Cuntz semigroup. Lecture notes from a course at the University of Münster, winter semester 2016-17.
	
\end{thebibliography}
\end{document}